\documentclass[11pt]{article}

\newcommand{\PaperTitle}{Subtree Counts and the Maximum Mean Subtree Order}
\newcommand{\PaperAuthor}{Jungang Chen, Xian'an Jin, Zhuo Li, and Hongxia Ma}
\newcommand{\PaperAffiliationXiamen}{School of Mathematical Sciences, Xiamen University, P. R. China}
\newcommand{\PaperAffiliationQinghai}{School of Mathematics and Statistics, Qinghai Minzu University, P. R. China}
\newcommand{\PaperAffiliationXinjiang}{Xinjiang Normal University, P. R. China}
\newcommand{\PaperEmailChen}{jgchen@stu.xmu.edu.cn}
\newcommand{\PaperEmailJin}{xajin@xmu.edu.cn}
\newcommand{\PaperEmailLi}{lzhuo@stu.xmu.edu.cn}
\newcommand{\PaperEmailMa}{hongxiama516@163.com}
\newcommand{\PaperDate}{September 2026}
\usepackage{amsmath,amssymb,amsthm,mathtools}
\usepackage{microtype}
\usepackage{enumitem}
\setlist{nosep,leftmargin=2em}
\usepackage[numbers,square,sort]{natbib}
\let\parencite\citep
\let\textcite\citet

\usepackage{hyperref}
\usepackage[nameinlink,noabbrev]{cleveref}
\usepackage{bookmark}
\usepackage{csquotes}

\hypersetup{
  colorlinks=true,
  linkcolor=black,
  citecolor=black,
  urlcolor=black,
  pdfborder={0 0 0}
}

\numberwithin{equation}{section}
\allowdisplaybreaks[3]

\newtheoremstyle{paperplain}
  {8pt}{8pt}{\itshape}{}
  {\bfseries}{.}{0.5em}{}
\newtheoremstyle{paperdefinition}
  {8pt}{8pt}{}{}
  {\bfseries}{.}{0.5em}{}
\newtheoremstyle{paperremark}
  {6pt}{6pt}{}{}
  {\itshape}{.}{0.5em}{}

\theoremstyle{paperdefinition}
\newtheorem{definition}{Definition}[section]
\newtheorem{lemma}[definition]{Lemma}
\newtheorem{proposition}[definition]{Proposition}
\newtheorem{theorem}[definition]{Theorem}
\newtheorem{corollary}[definition]{Corollary}
\theoremstyle{paperremark}
\newtheorem{remark}[definition]{Remark}

\newcommand{\lesub}{\mathrel{\le_{\mathrm{sub}}}}
\newcommand{\prert}{\mathrel{\preceq_{\mathrm{rt}}}}
\newcommand{\Autrt}{\operatorname{Aut}_{\mathrm{rt}}}
\newcommand{\Inj}{\operatorname{Inj}}
\newcommand{\Law}{\operatorname{Law}}

\newcommand{\PGW}{\operatorname{PGW}}

\newcommand{\proofstep}[1]{\par\medskip\noindent\textbf{#1}\par\nobreak\smallskip}

\crefname{definition}{definition}{definitions}
\crefname{lemma}{lemma}{lemmas}
\crefname{proposition}{proposition}{propositions}
\crefname{theorem}{theorem}{theorems}
\crefname{corollary}{corollary}{corollaries}
\crefname{remark}{remark}{remarks}

\hypersetup{
  pdftitle={\PaperTitle},
  pdfauthor={\PaperAuthor},
  pdfsubject={Subtree counts and extremal mean subtree order},
  pdfkeywords={mean subtree order, complete graph, stochastic domination, coupling, Cayley tree, Galton-Watson tree}
}

\title{\bfseries \PaperTitle}
\author{%
  Jungang Chen\thanks{\raggedright\PaperAffiliationXiamen.
    \href{mailto:\PaperEmailChen}{\texttt{\PaperEmailChen}}.}\quad
  Xian'an Jin\thanks{\raggedright\PaperAffiliationXiamen;
    \PaperAffiliationQinghai.
    \href{mailto:\PaperEmailJin}{\texttt{\PaperEmailJin}}.}\quad
  Zhuo Li\thanks{\raggedright\PaperAffiliationXiamen.
    \href{mailto:\PaperEmailLi}{\texttt{\PaperEmailLi}}.}\quad
  Hongxia Ma\thanks{\raggedright\PaperAffiliationXinjiang.
    Corresponding author.
    \href{mailto:\PaperEmailMa}{\texttt{\PaperEmailMa}}.}%
}
\date{\PaperDate}

\begin{document}
\maketitle

\begin{abstract}
For a finite simple graph $G$ of positive order, let $s_k(G)$ be the number of its $k$-vertex tree subgraphs. We prove that, among graphs of a fixed order $n$, the ratios $s_k(G)/s_k(K_n)$ form a nonincreasing sequence in $k$. It follows that the complete graph maximizes the mean subtree order: $\mu(G)\leq\mu(K_n)$, with equality if and only if $G$ is complete. This monotonicity theorem also establishes the $s_{n-1}$-to-$s_n$ ratio conjecture posed in recent work on extremal mean subtree order. The proof lifts a known coupling of rooted Cayley-tree shapes of consecutive orders to uniform labeled tree subgraphs of $K_n$ that are nested at every sample point. Event inclusion gives the count inequalities, and a double-sum identity gives the mean bound.
\end{abstract}


\section{Introduction}\label{sec:introduction}

For a finite simple graph $G$ of order $n\geq1$, a \emph{subtree} is a tree
that occurs as a not necessarily induced subgraph of $G$. Writing $s_k(G)$
for the number of subtrees with $k$ vertices, we study the mean subtree order
\[
  \mu(G)=\frac{\sum_{k=1}^n k s_k(G)}{\sum_{k=1}^n s_k(G)}.
\]
We measure subtree order by the number of vertices. For the same collection
of subtrees, the mean subtree size measured by edges, as in
\textcite{ChinEtAl2018}, is one less than the mean subtree order. This
difference does not affect inequalities or equality comparisons between
mean subtree values.

The behavior of the mean under edge addition has been a central source of
extremal questions. Chin, Gordon, MacPhee, and Vincent
\parencite[Conj.~7.4]{ChinEtAl2018} conjectured that adding any missing edge to
a connected graph increases its mean subtree size. If true, that assertion
would imply that $K_n$ uniquely maximizes the mean among connected graphs of
order $n$. Cameron and Mol disproved the edge-by-edge monotonicity conjecture
and proposed the weaker local conjecture that every noncomplete connected
graph has at least one missing edge whose addition increases the mean
\parencite{CameronMol2021}. Chen, Wei, and Lian and, independently, Cambie,
Chen, Hao, and Tokar proved that, for every positive integer $k$, there is a
graph and a choice of $k$ missing edges whose addition decreases the mean
subtree order \parencite{ChenWeiLian2023,CambieEtAl2024}.

The global extremal problem is logically distinct from this local edge-addition
problem. Proving the weaker local conjecture of Cameron and Mol would give a
path to the complete graph along which the mean increases, and hence would
imply global extremality. The global extremal conclusion alone does not
establish this local conjecture. Our argument compares subtree counts across
orders within a fixed graph.

Cambie, Jooken, and Wagner recently formulated the complete-graph upper bound
for the mean subtree order both for connected graphs and, separately, for all
graphs \parencite[Conjs.~1.1(ii) and 5.3]{CambieJookenWagner2025}. They also
conjectured, for $n\geq2$, the endpoint ratio inequality
comparing orders $n-1$ and $n$
\parencite[Conjs.~1.7 and 5.1]{CambieJookenWagner2025}. They proved that, if
this endpoint inequality holds for every finite simple graph, then the full
adjacent-order chain follows, and hence so does the complete-graph mean bound
\parencite[Lem.~5.7 and Thm.~5.6]{CambieJookenWagner2025}. Our principal result proves this
full chain directly: for every simple graph $G$ of order $n\geq1$,
\begin{equation}\label{eq:intro-ratio}
  1=\frac{s_1(G)}{s_1(K_n)}
  \geq\frac{s_2(G)}{s_2(K_n)}
  \geq\cdots\geq\frac{s_n(G)}{s_n(K_n)}\geq0.
\end{equation}
In particular, when $n\geq2$, the case comparing orders $n-1$ and $n$ is the ratio
inequality in Conjectures~1.7 and 5.1 of that work. A double-sum identity then turns
\eqref{eq:intro-ratio} into
\[
  \mu(G)\leq\mu(K_n),
\]
with equality precisely when $G$ is complete. This scope includes disconnected
graphs, as in their Conjecture~5.3, and is broader than the connected-graph
formulation in their Conjecture~1.1(ii).

The proof uses the root-preserving stochastic domination of
Poisson--Galton--Watson trees conditioned on consecutive finite orders,
as stated in Theorem~2.1 of Lyons, Peled, and Schramm
\parencite{LyonsPeledSchramm2008}. We verify that each conditioned law is
the rooted-shape distribution of a uniform labeled Cayley tree with a uniform
root. We then lift the resulting shape coupling to labeled trees. For each
compatible pair of rooted shapes, we fix a root-preserving embedding and apply
a uniformly random injection into $[n]$. Automorphism counts show that, after
the roots are forgotten, the two marginals are uniform over the actual tree
subgraphs of $K_n$ of the corresponding orders. Pointwise containment then
becomes event inclusion for the subtrees lying in $G$, which proves
\eqref{eq:intro-ratio}.

Our probability notation and the basic definitions of probability spaces,
random elements, distributions, independence, and conditional probability
follow Durrett \parencite[Sections~1.1--1.3, 2.1, and 4.1]{Durrett2019}.
Sections~\ref{sec:notation} and~\ref{sec:coupling} specify our finite-state
conventions, rooted-tree shapes, and couplings. The order relation on shapes
and the tree-growth theorem follow Lyons, Peled, and Schramm.

\section{Subtree Counts and Notation}\label{sec:notation}

Every graph $G$ in the extremal problem is finite and simple. The auxiliary
probabilistic tree introduced in Section~\ref{sec:coupling} may be infinite
before it is conditioned to have a fixed order. For concrete graphs $H$ and
$G$, we write $H\lesub G$ when $V(H)\subseteq V(G)$ and
$E(H)\subseteq E(G)$. We use $\prert$ for root-preserving embeddings between
unlabeled rooted-tree isomorphism classes and $\subseteq$ for set or event
inclusion. Thus $\lesub$ always compares actual graphs with specified vertex
and edge sets. For graph-valued random elements, we write
$X(\omega)\lesub Y(\omega)$ to indicate containment at the sample point
$\omega$.

Fix an integer $n\geq1$ and label the vertices of the complete graph by
$V(K_n)=[n]=\{1,\dots,n\}$.
Because all quantities considered below are invariant under graph
isomorphism, any graph of order $n$ may be relabeled to have vertex set
$[n]$. For $1\le k\le n$, let
\[
  \mathcal T_{n,k}
  =\{H:H\lesub K_n,\ H\text{ is a tree with }k\text{ vertices}\}.
\]
If $G$ is a simple graph on the same vertex set $[n]$, define
\[
  s_k(G)=\bigl|\{H\in\mathcal T_{n,k}:H\lesub G\}\bigr|.
\]
In particular, $s_k(K_n)=|\mathcal T_{n,k}|$. For $k=1$ there is one tree
on each one-element vertex set, so $s_1(K_n)=n$. For $k\geq2$, Cayley's
formula \parencite[Chap.~1]{Moon1970} gives
\[
  s_k(K_n)=\binom nk k^{k-2}
  \quad(2\le k\le n).
\]
The mean subtree order of $G$ is defined by
\[
  \mu(G)=
  \frac{\sum_{k=1}^{n}k s_k(G)}{\sum_{k=1}^{n}s_k(G)}.
\]
The denominator is always positive because $s_1(G)=n$.

\begin{definition}[Rooted labeled tree set]\label{def:rooted-labelled}
For $1\le k\le n$, let
\[
  \widehat{\mathcal T}_{n,k}
  =\{(H,v):H\in\mathcal T_{n,k},\ v\in V(H)\}.
\]
Thus, $\widehat{\mathcal T}_{n,k}$ is the set of all actual $k$-vertex tree subgraphs of $K_n$, each equipped with a distinguished root.
\end{definition}

For nonnegative integers $x$ and $k$, let $(x)_0=1$ and write
\[
  (x)_k=x(x-1)\cdots(x-k+1)\quad(k\geq1)
\]
for the falling factorial.
In particular, $(x)_k=0$ when $k>x$.

\subsection*{Probability conventions}

A probability space is a triple $(\Omega,\mathcal F,P)$, where $\mathcal F$
is a $\sigma$-field of subsets of $\Omega$ and $P$ is a countably additive
probability measure on $\mathcal F$. On a finite set, we take $\mathcal F$
to be its power set $2^\Omega$ and abbreviate the space to $(\Omega,P)$.
For a finite set $S$, an $S$-valued random element is a map $X:\Omega\to S$
whose inverse images are events. Its distribution, or law, is the measure
$\Law_P(X)$ on $2^S$ given by $\Law_P(X)(C)=P(X\in C)$ for $C\subseteq S$.
We omit the subscript when the underlying measure is clear. This is the
finite-state form of the random-element and distribution conventions in Durrett
\parencite[Sections~1.2--1.3]{Durrett2019}. For any probability measure $\alpha$
on a finite set, we write $\alpha(x)$ for $\alpha(\{x\})$.

A random element is uniform on a nonempty finite set $S$ if each element
has probability $1/|S|$. A statement holds \emph{almost surely} (a.s.) if the
event on which it holds has probability one. For events $A$ and $B$ with
\mbox{$P(B)>0$}, conditional probability is defined by
$P(A\mid B)=P(A\cap B)/P(B)$
\parencite[Section~4.1.1, Example~4.1.5]{Durrett2019}.

\section{A Coupling of Rooted-Tree Shapes}\label{sec:coupling}

For a positive integer $r$, let $\mathcal R_r$ denote the finite set of
root-preserving isomorphism classes of unlabeled rooted trees on $r$ vertices.
An element $\rho\in\mathcal R_r$ is a rooted-tree isomorphism class, which
we call a \emph{rooted-tree shape}. For positive integers $r$ and $s$, let
$\rho\in\mathcal R_r$ and $\sigma\in\mathcal R_s$. We write
$\rho\prert\sigma$ if and only if a representative $R_\rho$ is isomorphic
to a subtree $H$ of a representative $R_\sigma$ by a map that sends the root
of $R_\rho$ to the root of $R_\sigma$. Thus $H$ contains that root but need
not be all of $R_\sigma$. This definition is independent of the choice of
representatives.

\begin{definition}[Cayley-tree shape distribution]\label{def:nu}
Choose a labeled tree uniformly at random on $[r]$, choose a root uniformly at
random, and then forget all labels. Denote the resulting probability measure
on $\mathcal R_r$ by $\nu_r$.
\end{definition}

\begin{lemma}[Shape probability]\label{lem:shape-probability}
Let $\rho\in\mathcal R_r$, and let $a(\rho)=|\Autrt(R_\rho)|$ be the
order of the root-preserving automorphism group of any representative
$R_\rho$. Then
\begin{equation}\label{eq:nu-shape}
  \nu_r(\rho)=\frac{r!}{a(\rho)r^{r-1}}.
\end{equation}
\end{lemma}

\begin{proof}
If $r=1$, both sides of \eqref{eq:nu-shape} equal $1$. Suppose that $r\geq2$.
On the fixed label set $[r]$, the orbit--stabilizer theorem gives
$r!/a(\rho)$ distinct rooted labeled trees of rooted shape $\rho$. Cayley's
formula gives $r^{r-2}$ unrooted labeled trees, each with $r$ choices of root,
so there are $r^{r-1}$ rooted labeled trees in total. Taking the ratio yields
\eqref{eq:nu-shape}.
\end{proof}

\subsection*{The external theorem in its source form}

The only non-elementary probability result used in
Theorem~\ref{thm:shape-coupling} is Theorem~2.1 of Lyons, Peled, and Schramm.
We quote its statement and the definition of its random trees.

\begin{quote}
\emph{Let $T_n=T_n(\lambda)$ be a $\PGW(\lambda)$ tree conditioned to have
$n$ vertices, where $n\in\mathbb N_{+}\cup\{\infty\}$. We consider the values
of $T_n$ to be equivalence classes of rooted trees under isomorphisms that
preserve the root.}

\medskip
\noindent\textbf{Theorem 2.1.}
\emph{$T_{n+1}$ stochastically dominates $T_n$ for every
\mbox{$n\in\mathbb N_{+}$}.}
\end{quote}
The definition and theorem appear in
\parencite[Section~2, Thm.~2.1]{LyonsPeledSchramm2008}. The theorem compares
fixed finite orders $r$ and $r+1$; we do not use the source variable
$T_\infty$.

\subsection*{The probability model and its coupling}

\proofstep{Poisson--Galton--Watson tree.}
For $\lambda>0$, let $T$ be an unconditioned $\PGW(\lambda)$ genealogical
tree on a probability space with measure $P_\lambda$. From a single root, it is generated
recursively as follows.

For every vertex $v$, independently choose its number $Z_v$ of children with
the $\operatorname{Poisson}(\lambda)$ law
\parencite[Example~1.6.13]{Durrett2019}:
\[
  P_\lambda(Z_v=k)=e^{-\lambda}\frac{\lambda^k}{k!}
  \qquad(k=0,1,2,\dots).
\]
The descendant trees rooted at distinct children are therefore independent
and have the same unconditioned $\PGW(\lambda)$ law. We forget the order in
which the children are generated, so the resulting object is an unordered
rooted tree. Before conditioning, it may be infinite. The offspring construction agrees
with the Galton--Watson process in Durrett
\parencite[Section~4.3.4]{Durrett2019}. Here we retain the genealogical tree,
whereas that process records the number of individuals in each generation.

\proofstep{Conditioning on the number of vertices.}
Write $|T|$ for the total number of vertices of $T$, including the root.
For $m\geq1$, the source variable $T_m=T_m(\lambda)$ is a random element of
$\mathcal R_m$ with the conditional law of $T$ given $|T|=m$. Denote this
probability measure by
\begin{equation}\label{eq:conditioned-law-definition}
  P_{\lambda,m}:=\Law(T_m).
\end{equation}
Thus, for $\rho\in\mathcal R_m$,
\[
  P_{\lambda,m}(\{\rho\})
  =P_\lambda
    (T\cong_{\mathrm{rt}}R_\rho\mid |T|=m),
\]
where $\cong_{\mathrm{rt}}$ denotes root-preserving isomorphism. The source
identifies this conditional law with the shape distribution of a uniform
labeled tree with a uniform root. We verify the identification below:
\begin{equation}\label{eq:conditioned-law-roadmap}
  \Law(T_m)=P_{\lambda,m}=\nu_m.
\end{equation}
The conditioning is well defined because the Borel total-progeny formula
gives
\[
  P_\lambda(|T|=m)
  =\frac{(\lambda e^{-\lambda})^m m^{m-1}}{\lambda m!}
  =\frac{e^{-\lambda m}\lambda^{m-1}m^{m-1}}{m!}>0.
\]
When $\lambda>1$, these finite-size masses sum to the extinction probability,
and the remaining mass is carried by $|T|=\infty$; only fixed finite
conditioning events are used below. The definition of $T_m$ and the Borel
formula are given together in the source of the quoted theorem
\parencite[Section~2, Eq.~(2.1)]{LyonsPeledSchramm2008}.

\proofstep{Domination of rooted trees.}
In the terminology of Lyons, Peled, and Schramm, a rooted tree $(S,o_S)$
dominates a rooted tree $(R,o_R)$ if there is an isomorphism from $R$ onto a
subtree of $S$ that maps $o_R$ to $o_S$
\parencite[Section~1, after Thm.~1.2]{LyonsPeledSchramm2008}. Thus, in the
notation introduced at the beginning of this section, $\sigma$ dominates
$\rho$ precisely when $\rho\prert\sigma$.

\begin{definition}[Coupling and stochastic domination]\label{def:coupling}
Let $\alpha$ and $\beta$ be probability measures on finite sets $A$ and $B$,
respectively. A coupling of $\alpha$ and $\beta$ is a probability measure
$\pi$ on $A\times B$ such that
\begin{align*}
  \sum_{b\in B}\pi(a,b)&=\alpha(a)\quad(a\in A),\\
  \sum_{a\in A}\pi(a,b)&=\beta(b)\quad(b\in B).
\end{align*}
Equivalently, $\pi$ is the joint law of random elements $X$ and $Y$ on a
common probability space with measure $P$, with respective laws $\alpha$
and $\beta$. In this realization, $\pi(a,b)=P(X=a,Y=b)$, and the displayed
identities specify the marginals $\Law_P(X)$ and $\Law_P(Y)$.

Now suppose that $A$ and $B$ consist of rooted-tree shapes and are equipped
with the relation $\prert$. We say that $\beta$ \emph{stochastically dominates}
$\alpha$ if they have a coupling $(X,Y)$ such that
\[
  P(X\prert Y)=1.
\]
Equivalently, their coupling measure gives mass zero to every pair $(a,b)$
for which $a\not\prert b$. 
\end{definition}

This finite-state definition agrees with that of Lyons, Peled, and Schramm
\parencite[Section~1, after Thm.~1.2]{LyonsPeledSchramm2008}. Their theorem
supplies a coupling with root-preserving containment almost surely; it does
not concern independently sampled trees.

\begin{samepage}
\begin{theorem}[Coupling of consecutive Cayley-tree shapes]\label{thm:shape-coupling}
For every \mbox{$r\ge1$}, there exists a probability measure $\pi_r$ on
$\mathcal R_r\times\mathcal R_{r+1}$ such that
\begin{align}
  \sum_{\sigma\in\mathcal R_{r+1}}\pi_r(\rho,\sigma)
  &=\nu_r(\rho)
  &&(\rho\in\mathcal R_r),\label{eq:first-marginal}\\
  \sum_{\rho\in\mathcal R_r}\pi_r(\rho,\sigma)
  &=\nu_{r+1}(\sigma)
  &&(\sigma\in\mathcal R_{r+1}),\label{eq:second-marginal}\\
  \pi_r(\rho,\sigma)&=0
  &&\text{if }\rho\not\prert\sigma.\label{eq:support}
\end{align}
\end{theorem}
\end{samepage}

\begin{proof}
Fix $r\geq1$ and $\lambda>0$. By Theorem~2.1 quoted above,
\eqref{eq:conditioned-law-definition}, and Definition~\ref{def:coupling},
there are random rooted-tree shapes
$Y_r\in\mathcal R_r$ and $Y_{r+1}\in\mathcal R_{r+1}$ on one probability
space with measure $Q_r$ such that
\begin{align*}
  \Law_{Q_r}(Y_r)&=P_{\lambda,r},\\
  \Law_{Q_r}(Y_{r+1})&=P_{\lambda,r+1},\\
  Q_r(Y_r\prert Y_{r+1})&=1.
\end{align*}
It remains to identify the two conditioned laws with $\nu_r$ and
$\nu_{r+1}$.

For the unordered tree $T$ defined above, we claim that, for every
$m\geq1$ and every
$\rho\in\mathcal R_m$,
\begin{equation}\label{eq:pgw-shape}
  P_\lambda(T\cong_{\mathrm{rt}}R_\rho)
  =\frac{e^{-\lambda m}\lambda^{m-1}}{a(\rho)}.
\end{equation}
We prove the claim by induction on $m$. If $m=1$, the root has no children,
an event of probability $e^{-\lambda}$, in agreement with
\eqref{eq:pgw-shape}. Now suppose that $m>1$ and that the root of $R_\rho$ has
$k$ children. Every branch below the root has fewer than $m$ vertices, so the
induction hypothesis applies to it. After deleting the root, group the rooted
branches by root-preserving isomorphism type. Let the distinct types be
$\xi_1,\dots,\xi_s$, with respective multiplicities $b_1,\dots,b_s$, where
$\sum_jb_j=k$. For a rooted-tree isomorphism class $\xi$, let $|\xi|$ denote
the number of vertices in any representative.
\begin{samepage}
A root-preserving automorphism may act within each branch and may also permute
branches of the same type. Hence
\begin{equation}\label{eq:aut-decomposition}
  a(\rho)=\prod_{j=1}^s b_j!\,a(\xi_j)^{b_j}.
\end{equation}
\end{samepage}
The probability that the root has exactly $k$ children is
$e^{-\lambda}\lambda^k/k!$. Temporarily label these children. Since their
branches are independent and identically distributed, forgetting the
temporary labels contributes the multinomial factor
$k!/\prod_{j=1}^s b_j!$. The induction hypothesis therefore gives
\begin{align*}
  P_\lambda(T\cong_{\mathrm{rt}}R_\rho)
  &=\frac{e^{-\lambda}\lambda^k}{k!}
    \frac{k!}{\prod_{j=1}^s b_j!}
    \prod_{j=1}^s
    \left(
      \frac{e^{-\lambda|\xi_j|}
            \lambda^{|\xi_j|-1}}
           {a(\xi_j)}
    \right)^{b_j}\\
  &=\frac{e^{-\lambda m}\lambda^{m-1}}
          {\prod_{j=1}^s b_j!\,a(\xi_j)^{b_j}}
   =\frac{e^{-\lambda m}\lambda^{m-1}}{a(\rho)},
\end{align*}
where we used \eqref{eq:aut-decomposition} and
\begin{align*}
  m&=1+\sum_{j=1}^s b_j|\xi_j|,\\
  k+\sum_{j=1}^s b_j(|\xi_j|-1)&=m-1.
\end{align*}
This proves \eqref{eq:pgw-shape}.

Because $T\cong_{\mathrm{rt}}R_\rho$ already implies $|T|=m$, the definition
of conditional probability, \eqref{eq:pgw-shape}, and the Borel formula give
\begin{align}
  P_{\lambda,m}(\{\rho\})
  &=\frac{P_\lambda(T\cong_{\mathrm{rt}}R_\rho)}
          {P_\lambda(|T|=m)}\notag\\
  &=\frac{e^{-\lambda m}\lambda^{m-1}/a(\rho)}
          {e^{-\lambda m}\lambda^{m-1}m^{m-1}/m!}\notag\\
  &=\frac{m!}{a(\rho)m^{m-1}}
   =\nu_m(\rho).\label{eq:conditioned-shape}
\end{align}
The right-hand side is independent of $\lambda$. Since this holds for
every $\rho\in\mathcal R_m$, it proves \eqref{eq:conditioned-law-roadmap}:
\begin{equation}\label{eq:T-equals-nu}
  \Law(T_m)=P_{\lambda,m}=\nu_m.
\end{equation}

Define the joint mass of $(Y_r,Y_{r+1})$ by
\begin{equation}\label{eq:pi-joint-law}
  \pi_r(\rho,\sigma)
  :=Q_r(Y_r=\rho,\ Y_{r+1}=\sigma).
\end{equation}
By \eqref{eq:T-equals-nu}, the marginal identities are, for
$\rho\in\mathcal R_r$ and $\sigma\in\mathcal R_{r+1}$,
\begin{align*}
  \sum_{\sigma'\in\mathcal R_{r+1}}\pi_r(\rho,\sigma')
  &=Q_r(Y_r=\rho)=\nu_r(\rho),\\
  \sum_{\rho'\in\mathcal R_r}\pi_r(\rho',\sigma)
  &=Q_r(Y_{r+1}=\sigma)=\nu_{r+1}(\sigma).
\end{align*}
These are \eqref{eq:first-marginal} and \eqref{eq:second-marginal}.
Finally, if $\rho\not\prert\sigma$, then
\[
  \{Y_r=\rho,Y_{r+1}=\sigma\}
  \subseteq\{Y_r\not\prert Y_{r+1}\}.
\]
The event on the right has $Q_r$-probability zero, so
$\pi_r(\rho,\sigma)=0$. This is \eqref{eq:support} and completes the proof.
\end{proof}

\section{Lifting the Shape Coupling to Labeled Trees}\label{sec:lifting}

\begin{proposition}[A coupling of uniform labeled trees]\label{prop:label-lifting}
Fix $n$ and $1\le r<n$. There exist a finite probability space
$(\Omega_{n,r},P_{n,r})$ and random elements
$L_r:\Omega_{n,r}\longrightarrow\mathcal T_{n,r}$ and
$U_{r+1}:\Omega_{n,r}\longrightarrow\mathcal T_{n,r+1}$ such that:
\begin{enumerate}[label=\textup{(\roman*)}]
  \item $L_r(\omega)\lesub U_{r+1}(\omega)$ for every
    $\omega\in\Omega_{n,r}$;
  \item $L_r$ is uniformly distributed on $\mathcal T_{n,r}$;
  \item $U_{r+1}$ is uniformly distributed on $\mathcal T_{n,r+1}$.
\end{enumerate}
\end{proposition}

\begin{proof}
\proofstep{Step 1: Fixing representatives and embeddings.}
Choose a coupling measure $\pi_r$ supplied by
Theorem~\ref{thm:shape-coupling}.
For every $\rho\in\mathcal R_r$ and $\sigma\in\mathcal R_{r+1}$, fix
representatives $R_\rho=(V_\rho,E_\rho,o_\rho)$ and
$R_\sigma=(V_\sigma,E_\sigma,o_\sigma)$.
If $\pi_r(\rho,\sigma)>0$, then the support condition \eqref{eq:support} in
Theorem~\ref{thm:shape-coupling} gives
$\rho\prert\sigma$. By the representative-independence in the definition of
$\prert$, the fixed representatives $R_\rho$ and $R_\sigma$ also admit such an
embedding. We may therefore fix a root-preserving embedding
$\iota_{\rho,\sigma}:V_\rho\longrightarrow V_\sigma$ that maps $R_\rho$ isomorphically onto a subtree of $R_\sigma$. Since only
finitely many pairs are involved, fix one such embedding for each pair.

\proofstep{Step 2: Defining the finite probability space.}
Let
\[
  \Omega_{n,r}
  =\bigl\{(\rho,\sigma,f):
     \pi_r(\rho,\sigma)>0,
     \ f\in\Inj(V_\sigma,[n])\bigr\}.
\]
Here $\Inj(V_\sigma,[n])$ denotes the set of injections from $V_\sigma$ to
$[n]$.
Since $|V_\sigma|=r+1$, we have
$|\Inj(V_\sigma,[n])|=(n)_{r+1}$. For each point
$\omega=(\rho,\sigma,f)$, define
\begin{equation}\label{eq:point-mass}
  P_{n,r}(\{\omega\})
  =\frac{\pi_r(\rho,\sigma)}{(n)_{r+1}}.
\end{equation}
Summing over all $f$, $\rho$, and $\sigma$, and using
$\sum_{\rho,\sigma}\pi_r(\rho,\sigma)=1$, shows that the total mass is $1$.

\proofstep{Step 3: Defining the graph-valued random elements.}
For $\omega=(\rho,\sigma,f)$, first define the actual rooted tree
\[
  \widehat U_{r+1}(\omega)
  =\bigl(f(R_\sigma),f(o_\sigma)\bigr),
\]
where $f(R_\sigma)$ has vertex set $f(V_\sigma)$ and edge set
\[
  \bigl\{\{f(x),f(y)\}:\{x,y\}\in E_\sigma\bigr\}.
\]
Similarly, define
\[
  \widehat L_r(\omega)
  =\bigl((f\circ\iota_{\rho,\sigma})(R_\rho),
  f(\iota_{\rho,\sigma}(o_\rho))\bigr).
\]
Because the embedding preserves the root, the latter root equals
$f(o_\sigma)$. Let $U_{r+1}(\omega)$ and $L_r(\omega)$ be the actual
graphs obtained by forgetting the roots. By construction,
\begin{equation}\label{eq:pointwise-containment}
  L_r(\omega)\lesub U_{r+1}(\omega)
  \qquad(\omega\in\Omega_{n,r}).
\end{equation}

\proofstep{Step 4: Uniformity of the upper marginal.}
Fix $\widehat H=(H,h)\in\widehat{\mathcal T}_{n,r+1}$, and let
$\sigma_H\in\mathcal R_{r+1}$ be its unlabeled rooted-tree shape. There are
exactly $a(\sigma_H)$ root-preserving isomorphisms from the fixed
representative $R_{\sigma_H}$ to $\widehat H$. Therefore, by
\eqref{eq:point-mass} and the second-marginal identity
\eqref{eq:second-marginal},
\begin{align}
  P_{n,r}\bigl(\widehat U_{r+1}=\widehat H\bigr)
  &=\sum_{\rho\in\mathcal R_r}
    \pi_r(\rho,\sigma_H)
    \frac{a(\sigma_H)}{(n)_{r+1}}\notag\\
  &=\nu_{r+1}(\sigma_H)
    \frac{a(\sigma_H)}{(n)_{r+1}}\notag\\
  &=\frac{(r+1)!}{(r+1)^r(n)_{r+1}}
   =\frac1{\binom n{r+1}(r+1)^r}.
  \label{eq:upper-rooted-uniform}
\end{align}
The right-hand side is independent of $\widehat H$, so
$\widehat U_{r+1}$ is uniformly distributed on
$\widehat{\mathcal T}_{n,r+1}$.

\proofstep{Step 5: Uniformity of the lower marginal.}
Fix $\widehat H=(H,h)\in\widehat{\mathcal T}_{n,r}$, and let
$\rho_H\in\mathcal R_r$ be its unlabeled rooted-tree shape. For each
$\sigma$ satisfying $\pi_r(\rho_H,\sigma)>0$, we count the injections
$f$ for which $\widehat L_r(\omega)=\widehat H$, where
$\omega=(\rho_H,\sigma,f)$.
First, there are $a(\rho_H)$ root-preserving isomorphisms from
$R_{\rho_H}$ to $\widehat H$. Once one of them has been chosen, the set
$V_\sigma\setminus\iota_{\rho_H,\sigma}(V_{\rho_H})$ contains exactly one
vertex. That vertex may receive any label in $[n]\setminus V(H)$, giving
$n-r$ choices. Hence the number of such injections is exactly
$a(\rho_H)(n-r)$. By the first-marginal identity
\eqref{eq:first-marginal},
\begin{align}
  P_{n,r}\bigl(\widehat L_r=\widehat H\bigr)
  &=\sum_{\sigma\in\mathcal R_{r+1}}
    \pi_r(\rho_H,\sigma)
    \frac{a(\rho_H)(n-r)}{(n)_{r+1}}\notag\\
  &=\nu_r(\rho_H)
    \frac{a(\rho_H)(n-r)}{(n)_{r+1}}\notag\\
  &=\frac{r!}{r^{r-1}(n)_r}
   =\frac1{\binom nr r^{r-1}}.
  \label{eq:lower-rooted-uniform}
\end{align}
Here we used $(n)_{r+1}=(n)_r(n-r)$. The right-hand side is independent of
$\widehat H$, so $\widehat L_r$ is uniformly distributed on
$\widehat{\mathcal T}_{n,r}$.

\begin{samepage}
\proofstep{Step 6: Forgetting the roots.}
Fix $H\in\mathcal T_{n,r}$. The distinct choices $h\in V(H)$ give $r$
distinct elements $(H,h)\in\widehat{\mathcal T}_{n,r}$. If $r\ge2$, then
\eqref{eq:lower-rooted-uniform} and Cayley's formula give
\[
  P_{n,r}(L_r=H)
  =\frac{r}{\binom nr r^{r-1}}
  =\frac1{\binom nr r^{r-2}}
  =\frac1{|\mathcal T_{n,r}|}.
\]
\end{samepage}
If $r=1$, then every $H\in\mathcal T_{n,1}$ has a unique root, and
\eqref{eq:lower-rooted-uniform} gives
\[
  P_{n,1}(L_1=H)
  =\frac1{\binom n1\,1^0}
  =\frac1n
  =\frac1{|\mathcal T_{n,1}|}.
\]
For $H'\in\mathcal T_{n,r+1}$, summing
\eqref{eq:upper-rooted-uniform} over its $r+1$ possible roots gives
\begin{align*}
  P_{n,r}(U_{r+1}=H')
  &=\sum_{h'\in V(H')}
    P_{n,r}\bigl(\widehat U_{r+1}=(H',h')\bigr)\\
  &=\frac{r+1}{\binom n{r+1}(r+1)^r}\\
  &=\frac1{\binom n{r+1}(r+1)^{r-1}}
  =\frac1{|\mathcal T_{n,r+1}|}.
\end{align*}
Together with \eqref{eq:pointwise-containment}, these identities prove the
proposition.
\end{proof}

\section{Ratios of Subtree Counts}\label{sec:ratio}

The following is the principal theorem of the paper.

\begin{theorem}[Monotonicity of ratios]\label{thm:ratio-monotonicity}
For every integer $n\geq1$ and every simple graph $G$ on the vertex set $[n]$,
the sequence
\[
  q_k(G)=\frac{s_k(G)}{s_k(K_n)}
  \qquad(1\le k\le n)
\]
satisfies
\[
  1=q_1(G)\ge q_2(G)\ge\cdots\ge q_n(G)\ge0.
\]
\end{theorem}

\begin{proof}
The assertion is immediate when $n=1$. Assume $n\geq2$, and fix
$1\le r<n$. On the finite probability space constructed in
Proposition~\ref{prop:label-lifting}, define the events
\[
  A_r(G)=\{\omega\in\Omega_{n,r}:L_r(\omega)\lesub G\},
\]
\[
  B_{r+1}(G)=\{\omega\in\Omega_{n,r}:U_{r+1}(\omega)\lesub G\}.
\]
If $\omega\in B_{r+1}(G)$, then
\[
  L_r(\omega)\lesub U_{r+1}(\omega)\lesub G
\]
by \eqref{eq:pointwise-containment} and the transitivity of the subgraph relation. Consequently,
\[
  B_{r+1}(G)\subseteq A_r(G).
\]
The uniform marginal distributions of $L_r$ and $U_{r+1}$ give
\[
  P_{n,r}(A_r(G))
  =\frac{|\{H\in\mathcal T_{n,r}:H\lesub G\}|}{|\mathcal T_{n,r}|}
  =\frac{s_r(G)}{s_r(K_n)},
\]
and
\[
  P_{n,r}(B_{r+1}(G))
  =\frac{|\{H\in\mathcal T_{n,r+1}:H\lesub G\}|}{|\mathcal T_{n,r+1}|}
  =\frac{s_{r+1}(G)}{s_{r+1}(K_n)}.
\]
Taking probabilities in the event inclusion yields
\[
  \frac{s_{r+1}(G)}{s_{r+1}(K_n)}
  \le
  \frac{s_r(G)}{s_r(K_n)}.
\]
Applying this inequality for $r=1,\dots,n-1$ proves monotonicity. Moreover, $s_1(G)=s_1(K_n)=n$, so $q_1(G)=1$. Nonnegativity follows directly from the definition.
\end{proof}

\section{The Maximum Mean Subtree Order}\label{sec:mean}

\begin{lemma}[Double-sum identity]\label{lem:double-sum}
Let $n\geq1$. If $a_k>0$ and $q_k\ge0$ for $1\leq k\leq n$, and if
$\sum_{k=1}^n a_kq_k>0$, then
\begin{equation}\label{eq:double-sum}
  \frac{\sum_{k=1}^n k a_k}{\sum_{k=1}^n a_k}
  -
  \frac{\sum_{k=1}^n k a_kq_k}{\sum_{k=1}^n a_kq_k}
  =
  \frac{
    \displaystyle\sum_{1\le i<j\le n}
    a_i a_j(j-i)(q_i-q_j)
  }{
    \displaystyle
    \left(\sum_{k=1}^n a_k\right)
    \left(\sum_{k=1}^n a_kq_k\right)
  }.
\end{equation}
\end{lemma}

\begin{proof}
After passing to a common denominator, the numerator on the left-hand side is
\[
  \sum_{i,j=1}^n i a_i a_j(q_j-q_i).
\]
For each $i<j$, combine the terms indexed by $(i,j)$ and $(j,i)$ to obtain
\[
  i a_i a_j(q_j-q_i)+j a_j a_i(q_i-q_j)
  =a_i a_j(j-i)(q_i-q_j).
\]
Summing over all $i<j$ gives \eqref{eq:double-sum}.
\end{proof}

\begin{corollary}[Complete graphs maximize mean subtree order]\label{thm:main}
For every simple graph $G$ of order $n\geq1$,
\[
  \mu(G)\le\mu(K_n).
\]
Equality holds if and only if $G$ is complete, equivalently, if and only if
$G\cong K_n$.
\end{corollary}

\begin{proof}
If $n=1$, then every simple graph of order $1$ is isomorphic to $K_1$, and
$\mu(G)=\mu(K_1)=1$. Thus both the inequality and the equality characterization
hold. Suppose that $n\ge2$. After relabeling $G$ so that $V(G)=[n]$, set
$a_k=s_k(K_n)$ and $q_k=\frac{s_k(G)}{s_k(K_n)}$.
Then $a_k>0$, $q_k\ge0$, and $s_k(G)=a_kq_k$. Moreover,
\[
  \sum_{k=1}^n a_kq_k
  =\sum_{k=1}^n s_k(G)
  \geq s_1(G)=n>0.
\]
The hypotheses of Lemma~\ref{lem:double-sum} are therefore satisfied, and
\begin{align*}
  \mu(K_n)&=\frac{\sum_{k=1}^n k a_k}{\sum_{k=1}^n a_k},\\
  \mu(G)&=\frac{\sum_{k=1}^n k a_kq_k}{\sum_{k=1}^n a_kq_k}.
\end{align*}
By Theorem~\ref{thm:ratio-monotonicity}, $q_i-q_j\ge0$ whenever $i<j$.
Every summand on the right-hand side of \eqref{eq:double-sum} is therefore
nonnegative, so
\[
  \mu(K_n)-\mu(G)\ge0.
\]
If $G$ is complete, then all $q_k=1$, so equality holds. Conversely, if $G$
is not complete, then $q_1=1$ and
\[
  q_2=\frac{|E(G)|}{\binom n2}<1.
\]
The summand $a_1a_2(q_1-q_2)$ of \eqref{eq:double-sum}, corresponding to
$(i,j)=(1,2)$, is strictly positive. All remaining summands are nonnegative,
so $\mu(K_n)>\mu(G)$.
\end{proof}

\begin{remark}
By Cayley's formula, the mean subtree order of the complete graph can be written as
\[
  \mu(K_n)=
  \frac{
    n+\displaystyle\sum_{k=2}^n k\binom nk k^{k-2}
  }{
    n+\displaystyle\sum_{k=2}^n \binom nk k^{k-2}
  }.
\]
\end{remark}

\section{Concluding Remarks}\label{sec:conclusion}

Theorem~\ref{thm:ratio-monotonicity} compares the ratios of subtree counts
across all orders and yields the unique maximum in Corollary~\ref{thm:main}.
It does not establish the local conjecture of Cameron and Mol that every
noncomplete connected graph has a missing edge whose addition increases the
mean subtree order \parencite{CameronMol2021}.

\section*{Data Availability Statement}
Data sharing is not applicable to this article as no datasets were generated
or analyzed during this study.


\end{document}